\documentclass[12pt]{amsart}
\usepackage{amsmath,amssymb,amsthm}
\usepackage{latexsym}            
\usepackage{hyperref}
\usepackage{multirow}
\usepackage{color,colordvi,graphicx}
\newtheorem{theorem}{Theorem}
\newtheorem{lemma}[theorem]{Lemma}

\newtheorem{proposition}[theorem]{Proposition}

\theoremstyle{definition}
\newtheorem{definition}[theorem]{Definition}

\newtheorem{remark}[theorem]{Remark}

\newcommand{\sA}{{\mathcal A}}
\newcommand{\sB}{{\mathcal B}}

\newcommand{\sP}{{\mathcal P}}

\newcommand{\bN}{{\mathbb N}}

\newcommand{\bR}{{\mathbb R}}
\newcommand{\bC}{{\mathbb C}}

\newcommand{\bQ}{{\mathbb Q}}

\newcommand{\alphaCert}{{\bf alphaCertified}}
\newcommand{\alphaCertS}{{\bf alphaCertified~}}

\newcommand{\sK}{{\mathcal K}}
\newcommand{\sR}{{\mathcal R}}

\newcommand\Var{{\mathcal V}}

\newcommand\Null{{\rm null}}
\newcommand\Compute{{\bf ComputeConstants}}
\newcommand\AlgCertX{{\bf CertifySolns}}
\newcommand\AlgCertDist{{\bf CertifyDistinctSoln}}
\newcommand\AlgCertReal{{\bf CertifyRealSoln}}
\newcommand\AlgCertCount{{\bf CertifyCount}}

\title[alphaCertified: certifying solutions to polynomial systems]{alphaCertified: certifying solutions to polynomial systems}

\author{Jonathan D. Hauenstein}
\address{Jonathan D. Hauenstein\\
         Department of Mathematics\\
         Texas A\&M University\\
         College Station\\
         Texas \ 77843\\
         USA}
\email{jhauenst@math.tamu.edu}
\urladdr{\url{http://www.math.tamu.edu/~jhauenst}}
\thanks{Research of Hauenstein supported in part by the Fields Institute and NSF grant DMS-1114336.}

\author{Frank Sottile}
\address{Frank Sottile\\
         Department of Mathematics\\
         Texas A\&M University\\
         College Station\\
         Texas \ 77843\\
         USA}
\email{sottile@math.tamu.edu}
\urladdr{\url{http://www.math.tamu.edu/~sottile}}
\thanks{Research of both authors supported in part by NSF grants DMS-0915211 and
        DMS-0922866, and the Royal Swedish Academy of Sciences through the Institut
        Mittag-Leffler.}

\subjclass{65G20, 65H05}
\keywords{certified solutions, alpha theory, polynomial system, numerical
  algebraic geometry}
\begin{document}

\begin{abstract}
 Smale's $\alpha$-theory uses estimates related to the convergence of
 Newton's method to certify that Newton iterations
 will converge quadratically to solutions to a square polynomial system.
 The program \alphaCertS implements algorithms based on $\alpha$-theory to certify
 solutions of polynomial systems using both exact rational arithmetic and
 arbitrary precision floating point arithmetic.
 It also implements algorithms that certify whether a given point corresponds to a real
 solution, and algorithms to heuristically validate solutions to
 overdetermined systems.  Examples are presented to demonstrate the algorithms.
\end{abstract}

\maketitle
%
\section*{Introduction}\label{Sec:intro}

Current implementations of numerical homotopy algorithms~\cite{AG90,Morgan,SW05}
such as PHCpack~\cite{V99}, HOM4PS~\cite{HOM4PS}, Bertini~\cite{BHSW06}, and
NAG4M2~\cite{NAG4M2} routinely and reliably solve systems of polynomial equations with dozens
of variables having thousands of solutions.
Here, `solve' means `compute numerical approximations to solutions.'
In each of these software packages, the solutions are validated heuristically---often by
monitoring iterations of Newton's method.
This works well in practice, giving solutions that are acceptable in most applications.
However, a well-known shortcoming of numerical methods for computing approximate solutions
to systems of polynomials is that the output is not certified.
This restricts their use in some applications, including those in pure mathematics.
The program \alphaCertS is intended to remedy this shortcoming.

In the 1980's, Smale~\cite{S86} and others investigated the convergence of Newton's method,
developing 
$\alpha$-theory~\cite[Ch.~8]{BCSS}.
This refers to a computable positive constant $\alpha(f,x)$ depending upon a system
$f\colon\bC^n\to\bC^n$ of polynomials and a point $x\in\bC^n$ such that, if
\[
   \alpha(f,x)\ <\ \frac{13-3\sqrt{17}}{4}\  \approx\  0.157671\,,
\]
then iterations of Newton's method starting at $x$ will converge quadratically to a
solution to $f$, which is a point $\xi\in\bC^n$ with $f(\xi)=0$.
In principle, Smale's $\alpha$-theory provides certificates for validating numerical
computations with polynomials.

Current implementations of numerical homotopy algorithms do not incorporate
$\alpha$-theory to certify their output or their path-tracking.
There have been two projects which use fixed double precision and focus on certified
path-tracking.
Malajovich~\cite{M03} released the most recent version of his Polynomial System
Solver in 2003, which uses $\alpha$-theory to certify toric path-tracking algorithms, but
he states that ``[it] is actually not intended for an end user.''
Beltr\'an and Leykin~\cite{BL} have recently shown how to use $\alpha$-theory to certify
path-tracking, and hence the output of numerical homotopy algorithms.
While they demonstrate that certification can dramatically affect the speed of computation,
this is an important development, as certified path-tracking is necessary for applications
such as numerical irreducible decomposition~\cite{SVW01} or computing Galois
groups~\cite{LS}.
They are continuing this line of research.

We describe a program, \alphaCert, that implements elements of $\alpha$-theory to certify
numerical solutions to systems of polynomial equations using both exact
rational and arbitrary precision floating point arithmetic.
As it only certifies the output of a numerical computation, it avoids the
bottlenecks of certified tracking, while delivering some of its benefits.
Given a
square polynomial system $f:\bC^n\rightarrow\bC^n$, \alphaCertS uses Smale's $\alpha$-theory
to answer the following three questions for a finite set of points $X\subset\bC^n$:
\begin{enumerate}\label{Enum:Questions}
 \item
      From which points of $X$ will Newton's method converge quadratically to some
      solution to $f$?
 \item
      From which points of $X$ will Newton's method converge quadratically to distinct
      solutions to $f$?
 \item
       If $f$ is real
       ($\{\overline{f_1},\dotsc,\overline{f_n}\}=\{f_1,\dotsc,f_n\}$), from which points
       of $X$ will Newton's method  converge quadratically to real solutions to $f$?
\end{enumerate}
Often, a sharp upper bound $B$ on the number of roots to a square polynomial system $f$
is known.
Given a set of $B$ points, \alphaCertS can be used to certify that
iterations of Newton's method starting from each point in the set
converge quadratically to some solution to $f$ and
that these solutions are distinct.
This guarantees that each of the $B$
roots of $f$ can be approximated to arbitrary accuracy using Newton's method.
Moreover, \alphaCertS can certify how many of the $B$ solutions to $f$ are real
when $f$ is real.

A polynomial system $f:\bC^n\rightarrow\bC^N$ is {\it overdetermined} if $N > n$,
that is, if the number of polynomials exceeds the number of variables.
Dedieu and Shub \cite{DS00} studied Newton's method for overdetermined
polynomial systems and gave conditions which guarantee quadratic convergence of its
iterations.
Unlike square systems, the fixed points of this overdetermined Newton's method
need not be solutions.  For example, $x = 1$ is a fixed
point of Newton's method applied to
$f(x) = \left[\begin{array}{c} x \\ x-2 \end{array}\right]$.

The program \alphaCertS validates solutions to overdetermined systems.
Given a finite set $X\subset\bC^n$ and an overdetermined system, it generates two or more
random square subsystems, answers the three questions above for each, and compares the
results.
In particular, given $\delta > 0$, it can certify that, for a given approximate solution
to two or more random subsystems, the associated solutions all lie within a distance $\delta$
of each other.
For a given $\delta$, this heuristically validates solutions to
overdetermined systems.

In summary, \alphaCertS is novel in each of the following ways.
It implements algorithms from $\alpha$-theory
using either exact rational or arbitrary precision floating point arithmetic.
When using exact rational arithmetic with a square polynomial system, its implementation
of $\alpha$-theory is completely rigorous.
It certifiably determines if an approximate solution corresponds to a real solution,
which may be used to count the real solutions to a polynomial system, and
it uses $\alpha$-theory to obtain information on the roots of overdetermined systems.
The examples we give demonstrate the practicality of certification based on
$\alpha$-theory, and its viability as an alternative to exact symbolic methods, as the
certificates for square systems when using exact rational arithmetic are mathematical
proofs of computed results.

In Section~\ref{Sec:Alpha}, we review the concepts of $\alpha$-theory
utilized by \alphaCert.  Section~\ref{Sec:Square} presents the algorithms for
square polynomial systems while Section~\ref{Sec:OverDet} describes our approach
to overdetermined polynomial systems.  Implementation details are presented
in Section~\ref{Sec:Details} with examples presented in Section~\ref{Sec:Ex} verifying
some computational results in kinematics and generating evidence
for conjectures in enumerative real algebraic geometry.

%
\section{Smale's $\alpha$-theory}\label{Sec:Alpha}

We summarize key points of Smale's $\alpha$-theory for square polynomial systems that are
utilized by \alphaCert.
More details may be found in~\cite[Ch.~8]{BCSS}.

Let $f:\bC^n\rightarrow\bC^n$ be a system of $n$ polynomials in $n$ variables with
common zeroes $\Var(f):=\{\xi\in\bC^n\mid f(\xi) = 0\}$, and let $Df(x)$ be the Jacobian
matrix of the system $f$ at $x$.
Consider the map $N_f:\bC^n\rightarrow\bC^n$ defined by
\[
   N_f(x)\ :=\
      \begin{cases} x - Df(x)^{-1}f(x) & \hbox{if ~}Df(x)\hbox{~is invertible,}
       \\ x & \hbox{otherwise}. \end{cases}
\]
The point $N_f(x)$ is called the {\it Newton iteration of $f$ starting at $x$}.
For $k\in\bN$, let
\[
  N_f^k(x)\  :=\  \underbrace{N_f \circ \cdots \circ N_f(x)}_{k\hbox{~times}}
\]
be the $k^{th}$ Newton iteration of $f$ starting at $x$.

\begin{definition}\label{Defn:ApproxSoln}
 Let $f:\bC^n\rightarrow\bC^n$ be a polynomial system.
 A point $x\in\bC^n$ is an {\it approximate solution} to $f$ with {\it associated solution}
 $\xi\in\Var(f)$ if, for every $k\in\bN$,
 \begin{equation}\label{Eq:QuadConv}
   \|N_f^k(x) - \xi\|\ \leq\ \left(\frac{1}{2}\right)^{2^k - 1} \|x - \xi\| \,.
 \end{equation}
 That is, the sequence $\{N_f^k(x)\mid k\in\bN\}$ {\it converges quadratically} to $\xi$.
 Here, $\|\cdot\|$ is the usual hermitian norm on $\bC^n$, namely
 $\|(x_1,\dotsc,x_n)\|=(|x_1|^2+\dotsb+|x_n|^2)^{1/2}$.
\end{definition}

Smale's $\alpha$-theory describes conditions that imply a given point $x$ is
an approximate solution to $f$.
It is based on constants $\alpha(f,x)$, $\beta(f,x)$, and $\gamma(f,x)$.
If $Df(x)$ is invertible, these are
%
%
%
 \begin{eqnarray}
  \alpha(f,x) &:=& \beta(f,x) \gamma(f,x)\,, \nonumber\\
  \beta(f,x) &:=& \|x - N_f(x)\|\ =\ \|Df(x)^{-1} f(x)\|\,,\qquad\mbox{and} \nonumber\\
  \gamma(f,x) &:=& \sup_{k\geq2} \left\| \frac{Df(x)^{-1}
     D^kf(x)}{k!}\right\|^{\frac{1}{k-1}}\,.
    \label{Eq:gamma}
 \end{eqnarray}
If $x\in\Var(f)$ is such that $Df(x)$ is not invertible,
then we define $\alpha(f,x) := \beta(f,x) := 0$ and $\gamma(f,x) := \infty$.
Otherwise, if $x\notin\Var(f)$ and $Df(x)$ is not invertible,
then we define $\alpha(f,x) := \beta(f,x) := \gamma(f,x) := \infty$. \smallskip

In the formula~\eqref{Eq:gamma} for $\gamma(f,x)$, the $k^{th}$ derivative
$D^kf(x)$~\cite[Chap. 5]{L83} to $f$ is the symmetric tensor whose components are the
partial derivatives of $f$ of order $k$.
It is a linear map from the $k$-fold symmetric power $S^k\bC^n$ of $\bC^n$ to $\bC^n$.
The norm in~\eqref{Eq:gamma} is the operator norm of
$Df(x)^{-1} D^kf(x)\colon S^k\bC^n\to\bC^n$,
defined with respect to the norm on $S^k\bC^n$ that is dual to the standard unitarily
invariant  norm on homogeneous polynomials \cite{Kostlan},
\[
  \bigl\|\sum_{|\nu| = d} a_\nu x^\nu\bigr\|^2 \ :=\
  \sum_{|\nu| = d} |a_\nu|^2/ \tbinom{d}{\nu}\,,
\]
where $\nu=(\nu_1,\dotsc,\nu_n)$ is an exponent vector of non-negative integers with
$x^\nu=x_1^{\nu_1}\dotsb x_n^{\nu_n}$, $|\nu|=\nu_1+\dotsb+\nu_n$,
and $\tbinom{d}{\nu}=\frac{d!}{\nu_1!\dotsb\nu_n!}$ is the multinomial coefficient.

The following version of Theorem~2 from page 160 of~\cite{BCSS} provides a certificate
that a  point $x$ is an approximate solution to $f$.

\begin{theorem}\label{Thm:Alpha}
If $f:\bC^n\rightarrow\bC^n$ is a polynomial system and $x\in\bC^n$ with
\begin{equation}\label{Eq:alpha}
   \alpha(f,x)\ <\ \frac{13 - 3 \sqrt{17}}{4}\ \approx\ 0.157671\,,
\end{equation}
then $x$ is an approximate solution to $f$.
Additionally, $\|x - \xi\| \leq 2\beta(f,x)$ where $\xi\in\Var(f)$ is the associated solution to $x$.
\end{theorem}

\begin{remark}
 If $\alpha(f,x) \geq \frac{1}{4}$, then $x$ may not be an approximate solution to $f$.
 For example, for $f(x) = x^2$, if $x\neq 0$, then $x$ is not an approximate solution to
 $f$ yet $\alpha(f,x) = \frac{1}{4}$.
\end{remark}

For a polynomial system $f:\bC^n\rightarrow\bC^n$ and a point $x\in\bC^n$,
we say that $x$ is a {\it certified approximate solution} to $f$
if \eqref{Eq:alpha} holds.

Theorem~4 and Remark~6 of~\cite[Ch.~8]{BCSS} give a version of
Theorem~\ref{Thm:Alpha} that \alphaCertS uses to certify that two approximate solutions
have the same associated solution.

\begin{theorem}\label{Thm:AlphaRobust}
Let $f:\bC^n\rightarrow\bC^n$ be a polynomial system, $x\in\bC^n$ with $\alpha(f,x) <
0.03$ and $\xi\in\Var(f)$ the associated solution to $x$.  If $y\in\bC^n$ with
\[
   \|x - y\|\ <\ \frac{1}{20\gamma(f,x)}\,,
\]
then $y$ is an approximate solution to $f$ with associated solution $\xi$.
\end{theorem}

%
\subsection{Bounding higher order derivatives}\label{Sec:BoundGamma}

The constant $\gamma(f,x)$ encoding the behavior of the higher order derivatives
of $f$ at $x$ is difficult to compute, but it can be bounded above.
For a polynomial $g:\bC^n\rightarrow\bC$
of degree $d$, say $g = \sum_{|\nu| \leq d} a_\nu x^\nu$, define
\[
   \|g\|^2\ :=\ \sum_{|\nu| \leq d} |a_\nu|^2 \ \frac{\nu! (d - |\nu|)!}{d!}\,.
\]
Then $\|\cdot\|$ is the standard unitarily invariant norm on the homogenization of $g$.
For a polynomial system $f:\bC^n\rightarrow\bC^n$, define
\[
  \|f\|^2\ :=\ \sum_{i=1}^n \|f_i\|^2
   \qquad\mbox{where}\qquad
   f(x)\ =\ \left[\begin{array}{c} f_1(x) \\ \vdots \\ f_n(x) \end{array}\right]\,,
\]
and for a point $x\in\bC^n$, define
\[
  \|x\|_1^2\ :=\ 1+\|x\|^2\ =\ 1 + \sum_{i=1}^n |x_i|^2\,.
\]
Let $\Delta_{(d)}(x)$ be the $n \times n$ diagonal matrix with
\[
   \Delta_{(d)}(x)_{i,i}\ :=\ d_i^{1/2} \|x\|_1^{d_i-1}\,,
\]
where $d_i$ is the degree of $f_i$.
If $Df(x)$ is invertible, define
\[
   \mu(f,x) \ :=\ \max\{1, \|f\|\cdot\|Df(x)^{-1} \Delta_{(d)}(x)\|\}\,.
\]

The following version of Proposition~3 from \S I-3 of \cite{SS93}
gives an upper bound for $\gamma(f,x)$.

\begin{proposition}\label{Prop:BoundGamma}
 Let $f:\bC^n\rightarrow\bC^n$ be a polynomial system with $d_i = \deg f_i$
 and $D = \max d_i$.  If $x\in\bC^n$ such that $Df(x)$ is invertible, then
 \begin{equation}\label{Eq:GammaBound}
  \gamma(f,x)\ \leq\ \frac{\mu(f,x) D^{\frac{3}{2}}}{2 \|x\|_1}\,.
 \end{equation}
\end{proposition}

%
\section{Algorithms for square polynomial systems}\label{Sec:Square}

Let $f:\bC^n\rightarrow\bC^n$ be a square polynomial system and
$X = \{x_1,\dots,x_k\}\subset\bC^n$ be a set of points.
We describe the algorithms implemented in \alphaCertS
which answer the three questions posed in
the Introduction.  These algorithms are stated for a polynomial system
with complex coefficients, but are implemented for polynomial systems
with coefficients in $\bQ[\sqrt{-1}]$ using both exact and arbitrary precision arithmetic.

For each $i=1,\dots,k$, \alphaCertS first checks if $f(x_i)=0$.
If $f(x_i)\neq 0$, then
\alphaCertS determines if $Df(x_i)$ is invertible.
If it is, \alphaCertS computes $\beta(f,x_i)$ and
upper bounds for $\alpha(f,x_i)$ and $\gamma(f,x_i)$
using the following algorithm.

\begin{description}
  \item[Procedure $(\alpha,\beta,\gamma) = \Compute(f,x)$]
  \item[Input] A square polynomial system $f:\bC^n\rightarrow\bC^n$ and a point $x\in\bC^n$
          such that $Df(x)$ is invertible.
  \item[Output] $\alpha := \beta \cdot \gamma$, $\beta := \|Df(x)^{-1}f(x)\|$, and
    $\gamma$, where $\gamma$ is the upper bound for $\gamma(f,x)$ given in
    Proposition~\ref{Prop:BoundGamma}.
\end{description}

The next algorithm uses Theorem~\ref{Thm:Alpha}
to compute a subset $Y$ of $X$ containing points
that are certified approximate solutions to $f$.

\begin{description}
  \item[Procedure $Y = \AlgCertX(f,X)$]
  \item[Input] A square polynomial system $f:\bC^n\rightarrow\bC^n$ and a set
     $X = \{x_1,\dots,x_k\}\subset\bC^n$.
  \item[Output] A set $Y\subset X$ of approximate solutions to $f$.
  \item[Begin] \hskip -0.1in
  \begin{enumerate}
    \item Initialize $Y := \{\}$.
    \item For $j = 1,2,\dots,k$, if $f(x_j) = 0$, set $Y := Y \cup \{x_j\}$, otherwise,
              do the following if $Df(x_j)$ is invertible:
    \begin{enumerate}
      \item Set $(\alpha,\beta,\gamma) := \Compute(f,x_j)$.
      \item If $\alpha < \displaystyle\frac{13 - 3 \sqrt{17}}{4}$, set $Y := Y \cup \{x_j\}$.
    \end{enumerate}
  \end{enumerate}
  \item[Return] $Y$
\end{description}

As \alphaCertS uses the upper bound for $\gamma(f,x)$
of Proposition~\ref{Prop:BoundGamma}, it may fail to certify a legitimate
approximate solution $x$ to $f$.
In that case, a user may consider retrying after applying a few Newton
iterations to $x$.
The software \alphaCertS does not invoke an automatic
refinement to inputs that it does not certify.
This is because Newton iterations may have
unpredictable behavior (attracting cycles, chaos) when applied to points that are not in a
basin of attraction.
However, \alphaCertS does provide the functionality for the user to do this refinement.

Suppose that $x$ is an approximate solution to $f$ with associated solution $\xi$
such that $Df(\xi)$ is invertible.
Since $x$ is an approximate solution, $\beta(f,N_f^k(x))$ converges to zero.
Since $\gamma(f,x)$ is the supremum of a finite number of continuous functions of
$x$, $\gamma(f,N_f^k(x))$ is bounded.
In particular, $\alpha(f,N_f^k(x))$ converges to zero.

Given approximate solutions $x_1$ and $x_2$ to $f$ with associated solutions $\xi_1$ and
$\xi_2$, respectively, Theorems~\ref{Thm:Alpha} and~\ref{Thm:AlphaRobust}
can be used to determine if $\xi_1$ and $\xi_2$ are equal.
In particular,~if
\[
   \|x_1 - x_2\|\ >\ 2(\beta(f,x_1) + \beta(f,x_2))\,,
\]
then $\xi_1 \neq \xi_2$ by Theorem~\ref{Thm:Alpha}.
If on the other hand we have
\[
   \alpha(f,x_i)\ <\ 0.03
    \qquad\mbox{and}\qquad
    \|x_1 - x_2\|\ <\ \frac{1}{20\gamma(f,x_i)}
\]
for either $i = 1$ or $i= 2$, then $\xi_1 = \xi_2$ by Theorem~\ref{Thm:AlphaRobust}.
This justifies the following algorithm which determines if two approximate solutions
correspond to distinct associated solutions.

\begin{description}
  \item[Procedure $isDistinct = \AlgCertDist(f,x_1,x_2)$]
  \item[Input] A square polynomial system $f:\bC^n\rightarrow\bC^n$ and approximate solutions
      $x_1$ and $x_2$ to $f$ with associated solutions $\xi_1$ and $\xi_2$, respectively, such
      that $Df(\xi_1)$ and $Df(\xi_2)$ are invertible.
  \item[Output] A boolean $isDistinct$ that describes if $\xi_1\neq\xi_2$.
  \item[Begin]  Do the following:
    \begin{enumerate}
      \item[(a)] For $i = 1,2$, set $(\alpha_i,\beta_i,\gamma_i) := \Compute(f,x_i)$.
      \item[(b)] If $\|x_1 - x_2\| > 2(\beta_1 + \beta_2)$, {\bf Return} True.
      \item[(c)] If $\alpha_i < 0.03$ and $\|x_1 - x_2\| < \displaystyle\frac{1}{20\gamma_i}$,
          for either $i=1$ or $i=2$, {\bf Return} False.
      \item[(d)]\label{Item:UpdateDistinct} For $i = 1,2$, update $x_i := N_f(x_i)$ and return to (a).
    \end{enumerate}
\end{description}
This will halt, determining whether or not $\xi_1=\xi_2$
as $\beta(f,N_f^k(x_i))$ decreases quadratically with $k$,
while $\gamma(f,N_f^k(x_i))$ is bounded.

\subsection{Certifying real solutions}\label{Sec:Real}

A polynomial system $f:\bC^n\rightarrow\bC^n$ is {\it real} if
$\{\overline{f_1},\dotsc,\overline{f_n}\}=\{f_1,\dotsc,f_n\}$.
In that case, solutions to $f(x)=0$ are either real or occur in conjugate pairs.
Also, $N_f(\overline{x})=\overline{N_f(x)}$ for $x\in\bC^n$ so that
$N_f\colon\bR^n\to\bR^n$ is a real map.
Theorems~\ref{Thm:Alpha} and~\ref{Thm:AlphaRobust} can be used to
determine if an approximate solution of $f$ is associated to a real solution.
Let $x$ be an approximate solution to $f$ with associated solution $\xi$.
We do not assume that $x$ is real, for numerical continuation solvers
yield complex approximate solutions.
By assumption, $\overline{x}$ is also an approximate solution to $f$ with
associated solution $\overline{\xi}$.  If
\[
    \|x - \overline{x}\|\ >\ 2\left(\beta(f,x) + \beta(f,\overline{x})\right)\ =\ 4\beta(f,x)\,,
\]
then $\xi \neq \overline{\xi}$ by Theorem~\ref{Thm:Alpha} since
\[
  \|\xi - \overline{\xi}\|\ \geq\ \|x - \overline{x}\| - 4\beta(f,x)\ >\ 0\,.
\]
Consider the natural projection map $\pi_\bR: \bC^n\rightarrow\bR^n$ defined by
\[
   \pi_\bR(x)\ =\ \frac{x+\overline{x}}{2}\,.
\]
Since $\|x - \overline{x}\| = 2\|x - \pi_\bR(x)\|$, $\xi$ is not real if
\begin{equation}\label{Eq:NotReal}
    \|x - \pi_\bR(x)\|\ >\ 2\beta(f,x)\,.
\end{equation}

We have both a local and a global approach to show that $\xi$ is real.
For the local approach, Theorem~\ref{Thm:AlphaRobust}
implies that $\pi_\bR(x)$ is also an approximate solution to $f$ with associated solution
$\xi$ if
 \begin{equation}\label{Eq:Real}
   \alpha(f,x)\ <\ 0.03
   \qquad\mbox{and}\qquad
   \|x - \pi_\bR(x)\|\ <\ \frac{1}{20\gamma(f,x)}\,.
 \end{equation}
Since $N_f$ is a real map and $\pi_\bR(x)\in\bR^n$, this implies that $\xi\in\bR^n$.

We could also have showed that both $x$ and $\overline{x}$ correspond to the same solution
to deduce that $\xi = \overline{\xi}$.
If
\[
  \alpha(f,x)\ <\ 0.03
  \qquad\mbox{and}\qquad
  \|x - \overline{x}\|\ <\ \frac{1}{20\gamma(f,x)}\,,
\]
then Theorem~\ref{Thm:AlphaRobust} implies that $\xi = \overline{\xi}$.
This is more restrictive then~\eqref{Eq:Real}
since $\|x - \overline{x}\| = 2\|x - \pi_\bR(x)\|$.

When $\alpha(f,x) <\ 0.03$,~\eqref{Eq:NotReal} and~\eqref{Eq:Real}
yield closely related statements.  Since
\[
     \frac{5}{3}\beta(f,x)\ =\ \frac{5 \alpha(f,x)}{3 \gamma(f,x)}\
  <\ \frac{5\cdot 0.03}{3\gamma(f,x)}\ =\ \frac{1}{20\gamma(f,x)}\,,
\]
we know that $\xi$ is real if $\|x - \pi_\bR(x)\| \leq\ \frac{5}{3}\beta(f,x)$
and not real if $\|x - \pi_\bR(x)\| > 2\beta(f,x)$.

The following algorithm uses the local approach of~\eqref{Eq:NotReal} and~\eqref{Eq:Real}
to determine if an approximate solution
corresponds to a real associated solution.

\begin{description}
  \item[Procedure $isReal = \AlgCertReal(f,x)$]
  \item[Input] A real square polynomial system $f:\bC^n\rightarrow\bC^n$
       and an approximate solution $x\in\bC^n$ with associated solution $\xi$ such that
       $Df(\xi)$ is invertible.
  \item[Output] A boolean $isReal$ that describes if $\xi\in\bR^n$.
  \item[Begin] Do the following:
    \begin{enumerate}
      \item[(a)] Set $(\alpha,\beta,\gamma) := \Compute(f,x)$.
      \item[(b)] If $\|x - \pi_\bR(x)\| > 2\beta$, {\bf Return} False.
      \item[(c)] If $\alpha < 0.03$ and $\|x - \pi_\bR(x)\| < \displaystyle\frac{1}{20\gamma}$, {\bf Return} True.
      \item[(d)]\label{Item:UpdateReal} Update $x := N_f(x)$, and return to (a).
    \end{enumerate}
\end{description}

For the global approach to certifying real solutions, suppose that
we know {\em a priori} that $f$ has exactly $k$ solutions.
Suppose that $x_1,\dots,x_k$ are approximate solutions of $f$ with distinct associated
solutions.
If, for all $j\neq i$, $\overline{x_i}$ and $x_j$ also correspond to distinct solutions,
then $x_i$ and $\overline{x_i}$ must correspond to the same solution, which is
therefore real.
This global approach requires {\em a priori} knowledge about $\Var(f)$
as well as approximate solutions corresponding to each solution to $f$.
While it cannot be applied to all systems, it is
an alternative to the test based on $\gamma(f,x)$.

%
\subsection{Certification algorithm}\label{Sec:Certify}

For a given set of points $X$ and a polynomial system $f$,
\AlgCertX, \AlgCertDist, and \AlgCertReal~answer the three questions posed in the
Introduction.
We provide a sketch of the algorithm.

\begin{description}
  \item[Procedure $(A, D, R) = \AlgCertCount(f,X)$]
  \item[Input] A square polynomial system $f:\bC^n\rightarrow\bC^n$ and a finite set of
     points $X=\{x_1,\dots,x_\ell\}\subset\bC^n$ such that if $x_j$ is an approximate
     solution with associated solution $\xi_j$, then $Df(\xi_j)$ is invertible.
  \item[Output] A set $A\subset X$ consisting of certified approximate solutions to $f$, 
     a set $D \subset A$ consisting of points which have distinct associated solutions,
     and, if $f$ is a real map, a subset $R \subset D$ consisting of points
     which have real associated solutions.
  \item[Begin]\hskip -0.1in
  \begin{enumerate}
    \item Set $A := \AlgCertX(f,X)$.
    \item Set $n_A := |A|$ and enumerate the points in $A$ as $a_1,\dots,a_{n_A}$.
    \item For $j = 1,\dots,n_A$, set $s_j := \mbox{\it True}$.
    \item For $j = 1,\dots,n_A$ and for $k = j+1,\dots,n_A$, if $s_j$ and $s_k$ are \mbox{\it True},
      set $s_k := \AlgCertDist(f,a_j,a_k)$.
    \item Set $D := \{a_j\mid s_j = \mbox{\it True}\}$.
    \item Initialize $R := \{\}$.
    \item If $f$ is a real polynomial system, do the following:
    \begin{enumerate}
      \item Set $n_D := |D|$ and enumerate the points in $D$ as $d_1,\dots,d_{n_D}$.
      \item For $j = 1,\dots,n_D$, if $\AlgCertReal(f,d_j)$ is \mbox{\it True}, update $R := R \cup \{d_j\}$.
    \end{enumerate}
  \end{enumerate}
\end{description}

%
\section{Overdetermined polynomial systems}\label{Sec:OverDet}

When $N > n$, the polynomial system $f:\bC^n\rightarrow\bC^N$ is overdetermined.
Dedieu and Shub \cite{DS00} studied the overdetermined Newton's method whose iterates are defined by
 \begin{equation}\label{Eq:odNewton}
   N_f(x)\ :=\ x - Df(x)^\dagger f(x)\,,
 \end{equation}
where $Df^\dagger(x)$ is the Moore-Penrose pseudoinverse of $Df(x)$~\cite[\S~5.5.4]{GV96}
to determine conditions that guarantee quadratic convergence.
Since the fixed points of $N_f$ may not be solutions to the overdetermined
polynomial system $f$, this approach cannot certify solutions to overdetermined polynomial
systems.

We instead certify that points are
associated solutions to two or more random square subsystems using the algorithms
of Section~\ref{Sec:Square}.
An additional level of security may be added by certifying that,
for a given point which is an approximate solution to two or more random square
subsystems, the associated solutions lie within a given distance of each other.
As with the overdetermined Newton's method~\eqref{Eq:odNewton}, this
also cannot certify solutions to overdetermined polynomial systems,
which is still an open problem.

Let $R\colon\bC^N\to\bC^n$ be a linear map, considered as a matrix in $\bC^{n \times N}$.
Then $\sR(f)(x) = R\circ f(x)$ gives a square polynomial system
$\sR(f):\bC^n\rightarrow\bC^n$.
Since $\Var(f)\subset\Var(\sR(f))$ for any $R$,
we call $\sR(f)$ a {\it square subsystem} of $f$.
There is a nonempty Zariski open subset $\sA\subset\bC^{n\times N}$
such that for every $R\in\sA$ and every $x\in\Var(f)$,
$\Null~Df(x) = \{0\}$ if and only if $D\sR(f)(x)$ is invertible.
Moreover, for every $x\in\Var(\sR(f))\setminus\Var(f)$, $D\sR(f)(x)$ is invertible.
See \cite{SW05} for more on square subsystems $\sR(f)$.

Define $L = \{f(x)\mid x\in\bC^n\}\subset\bC^N$ which has dimension at most $n$ possibly passing through the origin.
A dimension-counting argument yields that
there is a nonemtpy Zariski open set $\sB\subset\sA\times\sA\subset\bC^{n\times N}\times\bC^{n\times N}$
such that, for every $(R_1,R_2)\in\sB$,
$K = \Null~R_1\cap\Null~R_2\subset\bC^N$ is a linear space
of dimension $\max\{N-2n, 0\}$ passing through the origin and
$K\cap L\subset\{0\}$.  In particular, if $\sR_i(f) = R_i \circ f$, then
\[
  \Var(\sR_1(f))\cap\Var(\sR_2(f))\ =\ \Var(f)\,.
\]

In addition, suppose that $x$ is an approximate solution to both
$\sR_1(f)$ and $\sR_2(f)$ with associated solutions $\xi_1$ and $\xi_2$, respectively.
For $k\in\bN$, define $x_{i,k} = N_{\sR_i(f)}^k(x)$ for $i = 1, 2$.
If $\xi_1 \neq \xi_2$, there exists $k\in\bN$ such that
\[
  \|x_{1,k} - x_{2,k}\|\ >\ 2(\beta(\sR_1(f),x_{1,k}) + \beta(\sR_2(f),x_{2,k}))\,,
\]
certifying that $\|\xi_1 - \xi_2\| > 0$.

If $\xi_1 = \xi_2$, then, for any $\delta > 0$, there exists $k\in\bN$ such that
\begin{equation}\label{Eq:odCertify}
   \|x_{1,k} - x_{2,k}\| + 2(\beta(\sR_1(f),x_{1,k}) + \beta(\sR_2(f),x_{2,k}))
    \ <\ \delta
\end{equation}
certifying that $\|\xi_1 - \xi_2\| < \delta$.
In particular, this certifies that the solutions $\xi_1$ and $\xi_2$ to $\sR_1$ and $\sR_2$ associated
to the common approximate solution $x$ lie within a distance $\delta$ of each other.
For $\delta \ll 1$,  this {\em heuristically} shows that $\xi_1 = \xi_2$.

In summary, if $x$ is a certified approximate solution to two different square
subsystems with distinct associated solutions, a certificate can be produced
demonstrating this fact.
Also (but not conversely), for any given tolerance $\delta>0$, a certificate can be
produced that the distance between the associated solutions to the two square subsystems
is smaller than $\delta$.

An additional test using the function residual could be added to this process.
The following lemma describes such a test.

\begin{lemma} \label{L:pie_in_sky}
Let $f:\bC^n\rightarrow\bC^N$ be an overdetermined polynomial system,
$R\in\bC^{n \times N}$, and $x$ be an approximate solution to $\sR(f) := R\circ f$
with associated solution $\xi$ such that $\alpha(\sR(f),x) \leq 0.0125$.
Then there exists $\epsilon > 0$ such that if there exists $y\in\bC^n$ satisfying
\[
   \|x - y\|\ \leq\ \frac{1}{40 \gamma(\sR(f),x)}
   \qquad\mbox{and}\qquad
   \|f(y)\|\ <\ \epsilon\,,
\]
then $\xi \in \Var(f)$.
\end{lemma}

\begin{proof}
Define $\nu = \displaystyle\frac{1}{40 \gamma(\sR(f),x)}$ and $B(x,\nu) = \{y\in\bC^N\mid \|x-y\|\leq \nu\}$.
We note that if $\gamma(\sR(f),x) = \infty$, since $B(x,\nu) = \{x\}$, it is easy to verify that we can take
\[
\epsilon = \left\{\begin{array}{ccl} 1 & ~~ & \hbox{if~}f(x) = 0, \\ \frac{\|f(x)\|}{2} & ~~ & \hbox{otherwise.} \end{array}\right.
\]
Hence, we can assume that $\gamma(\sR(f),x) < \infty$.  Since
\[
   \|x - \xi\|\ \leq\ 2 \beta(\sR(f),x)
    \ =\ \frac{2\alpha(\sR(f),x)}{\gamma(\sR(f),x)}
    \ \leq\ \frac{0.025}{\gamma(\sR(f),x)}
    \ =\ \nu\,,
\]
$\xi\in B(x,\nu)$.  Moreover, Theorem~\ref{Thm:AlphaRobust} yields that
$B(x,\nu) \cap \Var(\sR(f)) = \{\xi\}$.

Assume $\xi\notin\Var(f)$.  Since $\Var(f)\subset\Var(\sR(f))$, $B(x,\nu) \cap \Var(f) = \emptyset$.
In particular, $g(z) = \|f(z)\|$ is positive on the compact set $B(x,\nu)$.  Thus,
there exists $\epsilon > 0$ such that $\|f(y)\| \geq \epsilon$ for all $y \in B(x,\nu)$.
\end{proof}

\begin{remark}
  For Lemma~\ref{L:pie_in_sky} to give an algorithm, we would need a general bound for the minimum
  of a positive polynomial on a disk.
  In cases when such a bound is known,
  e.g., \cite{JP}, it is too small to be practical.
\end{remark}

%
\section{Implementation details for \alphaCert}\label{Sec:Details}

The program \alphaCertS is written in C and depends upon GMP \cite{GMP} and MPFR
\cite{MPFR} libraries to perform exact rational and arbitrary precision floating point
arithmetic.
When using rational arithmetic, all internal computations are {\em certifiable}.
Because of the bit length growth of rational numbers under algebraic computations,
\alphaCertS allows the user to select a precision and use floating point arithmetic
in that precision to facilitate computations.
Since floating point errors from  internal computations are not fully controlled,
\alphaCertS only yields a {\em soft certificate} when using
the floating point arithmetic option.
When the polynomial system is overdetermined,
\alphaCertS displays a message informing the user about what it has actually computed.

Three input files are needed to run \alphaCert.  These files contain the polynomial system,
the list of points to test, and the user-defined settings.
See \cite{HS10} for more details regarding exact syntax of these files.
The polynomial system is assumed to have rational complex coefficients
and described in the input file with respect to the basis of monomials.  That is, the user inputs
the coefficient and the exponent of each variable for each monomial term
in each polynomial of the polynomial system.

The set of points to test are assumed to have either rational coordinates
if using rational arithmetic or floating point coordinates if using floating point arithmetic.
When using floating point arithmetic, the points are inputted in the precision selected by the user.

The list of user-defined settings includes the choice between rational and floating point arithmetic,
the floating point precision to use for the basic computations if using floating point arithmetic,
and which certification algorithm to run.
The user can also define a value, say $\tau>0$, such that,
for each certified approximate solution, the associated solution
will be approximated to within $10^{-\tau}$ and printed to a file.

The specific output of \alphaCertS depends upon the user-defined settings.
In each case, an on-screen table summarizes the output as well as a
file that contains a human-readable summary for each point.
The other files created are machine-readable files that can be used
in additional computations.

Linear solving operations are performed using an $LU$ decomposition
and the spectral matrix norm is bounded above using the Frobenius norm.
This choice further worsens the approximation of $\gamma$
described in Proposition~\ref{Prop:BoundGamma},
which has two direct consequences on the performance of the algorithms.
First, this requires that the value of $\beta$ must be smaller in order
to certify approximate solutions in \AlgCertX.  Second, algorithms
\AlgCertDist~and \AlgCertReal~may need to utilize extra Newton iterations.
Apart from the added computational cost, the use of GMP and MPFR
allows \alphaCertS to still perform these computations even
when using such an approximation of $\gamma$.

When using rational arithmetic, \alphaCertS avoids taking square roots
when testing the required inequalities.  When using floating point arithmetic,
as an effort to control the floating point errors, the internal working precision
is increased when updating the point via a Newton
iteration, for instance in Step~(d) of \AlgCertDist~and Step~(d) of \AlgCertReal.

The software \alphaCertS determines if a square polynomial system $f$ in $n$ variables is
real using two tests.
The first test determines if the coefficients of $f$ are real.
The second selects a pseudo-random point $y\in\bQ^n$ and determines if
$\{f_1(y),\dots,f_n(y)\} = \{\overline{f_1(y)},\dots,\overline{f_n(y)}\}$.

The user either instructs \alphaCertS to bypass all tests and declare that $f$ is
real, or which tests to use.
If all tests fail, then \alphaCertS bypasses the real certification.
Otherwise, for each approximate solution $x$ with associated solution $\xi$,
\alphaCertS determines if there exists a real approximate solution
that also corresponds to $\xi$.
If the user incorrectly identified $f$ as real, then $\xi$ may not be real.
Therefore, \alphaCertS displays a message informing
the user about what it actually has certified.

For an overdetermined polynomial system $f$, \alphaCertS only checks
to see if all of the coefficients of $f$ are real.  In this case,
\alphaCertS randomizes $f$ using real matrices to obtain real square subsystems.

%
\section{Computational examples}\label{Sec:Ex}

We used \alphaCertS to study four  polynomial systems whose
number of real solutions is relevant.
Two are from kinematics and two are from enumerative geometry.
All involve polynomial systems that are not easily solved using certified methods from
symbolic computation.
The files used in the computations,
as well as instructions for their use, are found
on our website~\cite{HS10}.
Computations of Sections~\ref{Sec:Lines}
and~\ref{Sec:Schubert} used nodes of the Brazos cluster~\cite{Brazos}
that consist of two 2.5 GHz Intel Xeon E5420 quad-core processors.


\subsection{Stewart-Gough platform}

The Stewart-Gough platform is a parallel manipulator in which six variable-length
actuators are attached between a
fixed frame (the ground) and a moving frame (the platform)~\cite{Go57,St65}.
Each position of the platform uniquely determines the lengths of the six actuators.
However, the lengths of the actuators do not uniquely determine the position and
orientation of the platform, as there are typically several assembly modes, called
{\it positions}.

A generic platform with generic actuator lengths has 40 complex assembly modes.
Dietmaier~\cite{Dietmaier} used a continuation method to find a platform and leg
lengths for which all 40 positions are real.
While his formulation as a system of polynomial equations and conclusions about their
solutions being real have been reproduced numerically (this is a problem in Verschelde's
test suite~\cite{WWW_VerStew}), these computations only give a heuristic verification of
Dietmaier's result.

We modified the polynomial system from Verschelde's test suite, which uses the
parameters obtained by Dietmaier, by converting the
floating point numbers to rational numbers.  We then ran PHCpack~\cite{V99} on the
resulting polynomial system to obtain 40 numerical solutions to the system,
each of which it identified as real.
After converting the floating point coordinates of the solutions
to rational numbers, we ran \alphaCertS using these rational polynomials and
rational points.  It verified that these 40 points correspond to distinct real solutions.
This gives a rigorous mathematical proof of Dietmaier's result.


\subsection{Four-bar linkages}\label{Sec:NinePt}

A {\it four-bar linkage}   is a planar linkage consisting of a triangle with two of its
vertices connected to two bars, whose other endpoints are fixed in the plane.
The base of the triangle, the two attached bars, and the implied bar between the two fixed
points are the four bars.
\[
   \includegraphics{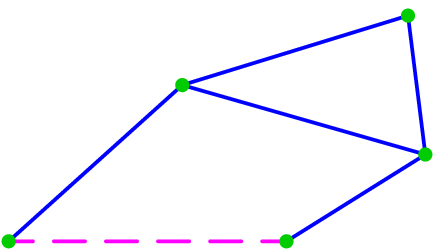}
\]
A general linkage has a one-dimensional constrained motion during which the joints
may rotate, and the curve traced by the apex of the triangle is its
{\it workspace curve}.

The {\it nine-point path synthesis problem} asks for the four-bar linkages whose workspace
curve contains nine given points.
Morgan, Sommese, and Wampler~\cite{MSW} used homotopy continuation to solve a polynomial
system which describes the four-bar linkages whose workspace
curves pass through nine given points.
They found that for nine points $\sP = \{P_0,\dots,P_8\}\subset\bC^2$ in general position, there are
8652 isolated solutions.
Due to a two-fold symmetry, there are 4326 distinct four-bar linkages which appear
in 1442 triplets, called Roberts cognates.
We used \alphaCertS to produce a soft certificate that the polynomial system
has at least 8652 isolated solutions
and, for a specific set of nine real points, certified the number of real
solutions among these 8652 solutions.

If $\sP\subset\bR^2$, the formulation of~\cite{MSW}
is not a real polynomial system.
The usual approach of writing the variables using real and imaginary parts gives a
real polynomial system $f_\sP$ consisting of four quadratic and eight quartic
polynomials.
For nine points $\sP = \{P_0,\dots,P_8\}\subset\bC^2$,
the polynomial system $f_\sP:\bC^{12}\rightarrow\bC^{12}$
depends upon the variables
\[
   \{a_1,\,a_2,\,n_1,\,n_2,\,x_1,\,x_2,\,b_1,\,b_2,\,m_1,\,m_2,\,y_1,\,y_2\}\,.
\]
Define the complex numbers
 \begin{align*}
  a &= a_1 + \sqrt{-1}\cdot a_2,&  n &= n_1 + \sqrt{-1}\cdot n_2,&  x &= x_1 + \sqrt{-1}\cdot x_2, \\
  b &= b_1 + \sqrt{-1}\cdot b_2,&  m &= m_1 + \sqrt{-1}\cdot m_2,&  y &= y_1 + \sqrt{-1}\cdot y_2,
 \end{align*}
whose complex conjugates are
$\overline{a},\overline{n},\overline{x},\overline{b},\overline{m},\overline{y}$, respectively.
These correspond to the variables used in the formulation in \cite{MSW}.
The four quadratic polynomials of $f_\sP$ are
\[
  \begin{array}{cc}
    f_1\ =\ n_1 - a_1x_1 - a_2x_2\,, & f_2\ =\ n_2 + a_1x_2 - a_2x_1\,, \\
    f_3\ =\ m_1 - b_1y_1 - b_2y_2\,, & f_4\ =\ m_2 + b_1y_2 - b_2y_1\,.
  \end{array}
\]
The eight quartic polynomials depend
upon the displacements from $P_0$ to the other points $P_j$.
For $j = 1,\dots,8$, define $Q_j := (Q_{j,1},Q_{j,2}) = P_j - P_0$ and
write each displacement $Q_j$ using {\em isotropic coordinates}, namely
$(\delta_j,\overline{\delta}_j)$ where
\[
   \delta_j = Q_{j,1} + \sqrt{-1}\cdot Q_{j,2}
   \hskip 0.1in \hbox{and} \hskip 0.1in
   \overline{\delta}_j = Q_{j,1} - \sqrt{-1}\cdot Q_{j,2}.
\]

For $j = 1,\dots,8$, the quartic polynomial $f_{4+j}$ of $f_\sP$ is
\[
  f_{4+j}\ :=\ \gamma_j \overline{\gamma}_j + \gamma_j\gamma_j^0 + \overline{\gamma}_j\gamma_j^0
\]
where
\[
  \begin{array}{ccccc}
    \gamma_j := q_j^x r_j^y - q_j^y r_j^x, & \hbox{~~~} & \overline{\gamma}_j := r_j^x p_j^y - r_j^y p_j^x, & \hbox{~~~} &
    \gamma_j^0 := p_j^x q_j^y - p_j^y q_j^x
  \end{array}
\]
and
\[
  \begin{array}{ccccc}
   p_j^x := \overline{n} - \overline{\delta}_j x, & \hbox{~~} & q_j^x := n - \delta_j \overline{x}, & \hbox{~~} &
   r_j^x := \delta_j(\overline{a}-\overline{x}) + \overline{\delta}_j(a-x) - \delta_j\overline{\delta}_j, \\    \rule{0pt}{16pt}
   p_j^y := \overline{m} - \overline{\delta}_j y, & \hbox{~~} & q_j^y := m - \delta_j \overline{y}, & \hbox{~~} &
   r_j^y := \delta_j(\overline{b}-\overline{y}) + \overline{\delta}_j(b-y) - \delta_j\overline{\delta}_j.
  \end{array}
\]

We first certified that, for nine randomly selected points in the complex plane,
the resulting polynomial system has at least 8652 isolated solutions.
Since the displacements $Q_j$ define the polynomial system,
we choose them to be points of $\bQ[\sqrt{-1}]^2$
with each coordinate having unit modulus of the form
 \[
   \frac{t^2-1}{t^2+1}\ +\ \sqrt{-1}\cdot\frac{2t}{t^2+1}
 \]
where $t$ was a quotient of two ten digit random integers.
We used regeneration \cite{HSW10} in Bertini \cite{BHSW06}
to compute 8652 points that were {\em heuristically}
within $10^{-100}$ of an isolated solution for $f_\sP$.
Then, \alphaCertS produced a soft certificate using
256-bit precision that these 8652 points are
approximate solutions to $f_\sP$ with distinct associated solutions.

We next certified the number of real solutions
for a specific set of nine real points, namely Problem 3 of \cite{MSW}.
The nine real points are listed in Table 2 of \cite{MSW}, which,
for convenience, we list the values of $\delta_j$ in Table~\ref{Tab:NinePts}.
\begin{table}[htb]
  \caption{Values of $\delta_j$ for Problem 3 of \cite{MSW}}\label{Tab:NinePts}
  \begin{tabular}{|c|l|}
  \hline
  $j$ & \multicolumn{1}{c|}{$\delta_j$} \\
  \hline
  1 & \hskip 0.125in $0.27+0.1\sqrt{-1}$\rule{0pt}{12pt}\\
  \hline
  2 & \hskip 0.125in $0.55+0.7\sqrt{-1}$\rule{0pt}{12pt}\\
  \hline
  3 & \hskip 0.125in $0.95+\sqrt{-1}$\rule{0pt}{12pt}\\
  \hline
  4 & \hskip 0.125in $1.15+1.3\sqrt{-1}$\rule{0pt}{12pt}\\
  \hline
  5 & \hskip 0.125in $0.85+1.48\sqrt{-1}$\rule{0pt}{12pt}\\
  \hline
  6 & \hskip 0.125in $0.45+1.4\sqrt{-1}$\rule{0pt}{12pt}\\
  \hline
  7 & $-0.05+\sqrt{-1}$\rule{0pt}{12pt}\\
  \hline
  8 & $-0.23+0.4\sqrt{-1}$\rule{0pt}{12pt}\\
  \hline
  \end{tabular}
\end{table}
Since the points are real, $\overline{\delta}_j$ is the conjugate of $\delta_j$.
We used parameter continuation in Bertini to
solve the resulting polynomial system starting from the
8652 solutions to the polynomial system
solved in the first test.  This generated a list of 8652 points
which \alphaCertS soft certified using 256-bit precision
to be approximate solutions that have distinct associated solutions of which 384 are real.
In particular, this confirms the results reported in
Table 3 of \cite{MSW} for Problem 3, namely,
that $64=384/6$ of the 1442 mechanisms are real.

Figure~\ref{F:four-bar} shows three of the 64 real mechanisms that solve
this synthesis problem, together with their workspace curves.
\begin{figure}[htb]
   \includegraphics[height=4.5cm]{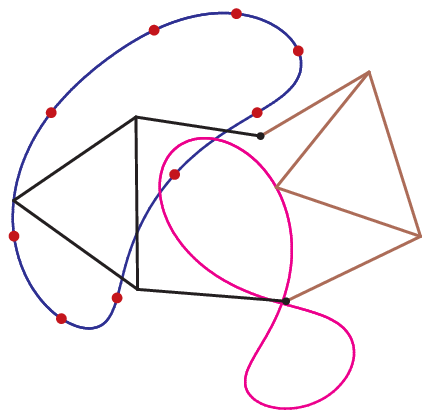}\qquad
   \includegraphics[height=4.5cm]{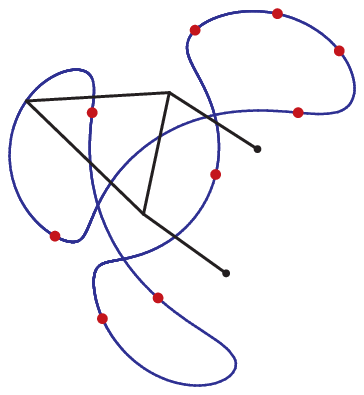}\qquad
   \includegraphics[height=4.5cm]{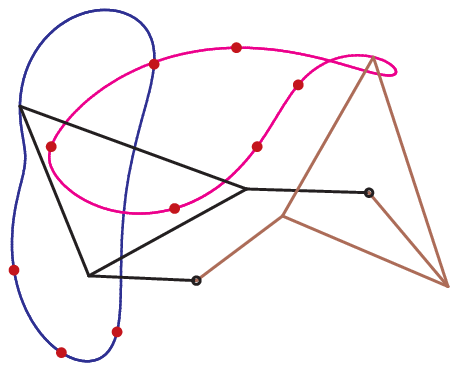}
\caption{Three solutions.}
\label{F:four-bar}
\end{figure}
The first has two assembly modes with the workspace curve of one mode a simple closed
curve that contains the nine target points.
This mechanism is the only viable mechanism among the 64 real mechanisms.
The second has only one assembly mode, but its workspace curve is convoluted and does not
meet the target points in a useful order.
The third has two assembly modes, and each only reaches a proper subset of the target
points.

%
\subsection{Lines, points, and conics}\label{Sec:Lines}

We consider geometric problems of plane conics in $\bC^3$ that meet $k$ points and $8-2k$
lines for $k=0,\dotsc,4$.
When the points and lines are general, the numbers of plane conics are known
and presented in Table~\ref{T:plc}.
\begin{table}[htb]
\caption{Numbers of plane conics}
 \begin{tabular}{|r||r|r|r|r|r|}\hline
                $k$ & 4 & 3 & 2 & 1  & 0 \\\hline
   Number of conics & 0 & 1 & 4 & 18 & 92\\\hline
 \end{tabular}
\label{T:plc}
\end{table}

This problem is from a class of problems in enumerative geometry---counting rational
curves---that has been of great interest in recent years~\cite{FuPa}.
For problems of enumerating rational curves of degree $d$ in the plane that interpolate
$3d-1$ real points, Welschinger~\cite{W} defined an invariant $W_d$ which is a lower bound
on the number of real rational curves, and work of Mikhalkin~\cite{Mi} and of Itenberg,
Kharlamov, and Shustin showed that $W_d$ is positive and eventually found a
formula for it~\cite{IKS}.

We used \alphaCertS to investigate the possible numbers of real solutions to these
problems of conics when their input data (points and lines) are real.
Of particular interest is the minimum number of solutions that are real.
Our experimental data suggests that when $k=1$ at least two of the solutions will be real,
and it shows that for $k=0,2$, it is possible to have no real solutions.

This experiment computed random instances of the problem.
The coordinates of points were taken to be the quotient of two random ten digit
integers, and the real lines were taken to be lines through two such points.
The resulting polynomial system was square.
Each real instance was solved by Bertini \cite{BHSW06} using a straight line
parameter homotopy starting with a fixed random complex instance
(see \cite{SW05} for more details).
This gave points that were {\em heuristically} within
$10^{-75}$ of each isolated solution.
Then \alphaCertS used 256-bit precision to softly certify that the points computed by
Bertini were approximate solutions corresponding to distinct solutions, and to count the
number of real solutions.
Since enumerative geometry provides the generic root count,
this yields a post-processing certificate that Bertini has indeed computed
an approximate solution corresponding to each solution to the polynomial system.

In every instance that Bertini successfully tracked every path,
the heuristic results of Bertini matched the certified results of \alphaCert.
Out of the over 1,450,000,000 paths tracked, 76 paths were truncated
by Bertini due to a fail-safe measure.  Thirty-two paths were truncated since
they needed more than the fail-safe limit of 10,000 steps along the path.
Each of these paths were successfully tracked when the limit was raised to 25,000 steps.
Forty-four paths were truncated since the adaptive precision tracking
algorithm \cite{BHSW08,BHSW09,BHS10} requested to use
more than the fail-safe limit of 1024-bit precision.
Each of these paths were successfully tracked when the fail-safe limit was raised to
1284-bit precision.

The first interesting case is when $k=2$ and there are
four conics meeting two points and four lines.
We solved 500 random real instances using the Brazos cluster.
Each instance took an average of $0.7$ seconds for Bertini to solve
and $0.1$ seconds for \alphaCertS to certify the results.
We found that there can be $0$, $2$, or $4$ real solutions.
Table~\ref{Tab:Conic24} presents the frequency distribution
of these 500 instances for this case.
{\scriptsize
\begin{table}[htb]
  \caption{Frequency distribution for conics through two points and four lines}\label{Tab:Conic24}
  \begin{tabular}{|c|c|c|c||c|}
  \hline
  \# real & 0 & 2 & 4 & total \\
  \hline
  frequency & 12 & 221 & 267 & 500 \\
  \hline
  \end{tabular}
\end{table}}

When $k=1$, there are 18 conics meeting a point and six lines in $\bC^3$.
We solved 1,000,000
random real instances using the Brazos cluster.
Each instance took an average
of $1.6$ seconds for Bertini to solve and
an average of $0.1$ seconds for \alphaCertS to certify the results.
Every real instance
that we computed had at least 2 real solutions.
Table~\ref{Tab:Conic18} presents the frequency distribution
of these 1,000,000 instances for this case.
{\scriptsize
\begin{table}[htb]
  \caption{Frequency distribution for conics through a point and six lines}\label{Tab:Conic18}
  \begin{tabular}{|c|c|c|c|c|c|c|c|c|c|c||c|}
  \hline
  \# real & 0 & 2 & 4 & 6 & 8 & 10 & 12 & 14 & 16 & 18 & total \\
  \hline
  frequency & 0 & 3281 & 21984 & 88813 & 193612 & 261733 & 226383 & 137074 & 53482 & 13638 & 1000000 \\
  \hline
  \end{tabular}
\end{table}
}

To compare the performance of \alphaCertS to  symbolic methods, we computed 40,000 instances of the
conic problem with $k=1$ using Singular~\cite{SINGULAR} to compute an eliminant that satisfies the Shape
Lemma~\cite{BMMT} and Maple to count the number of real roots of the eliminant, which is a standard
symbolic method to determine the number of real solutions to a zero-dimensional system of polynomial
equations.
The coordinates of points were taken to be rational numbers $p/q$ where $p,q$ were integers with
$|p|<4000$ and $0<q<1000$.
Each computation took approximately 661 seconds on a single node of a server with four
six-core AMD Opteron 8435 processors and 64 GB of memory.
Table~\ref{Tab:Conic18.symb} presents the frequency distribution
of these 40,000 instances for this case.
{\scriptsize
\begin{table}[htb]
  \caption{Frequency distribution for conics through a point and six lines}\label{Tab:Conic18.symb}
  \begin{tabular}{|c|c|c|c|c|c|c|c|c|c|c||c|}
  \hline
  \# real & 0 & 2 & 4 & 6 & 8 & 10 & 12 & 14 & 16 & 18 & total \\
  \hline
  frequency &  0 & 146 & 892 & 3558 & 7739 & 10575 & 8965 & 5488 & 2089 & 548 & 40000\\
  \hline
  \end{tabular}
\end{table}
}

Finally, when $k=0$, there are 92 plane conics meeting eight general lines in $\bC^3$.
We solved 15,662,000 random real instances using the Brazos cluster.
On average, each instance took $8.8$ seconds for Bertini to solve and
$0.7$ seconds for \alphaCertS to certify the results.
Table~\ref{Tab:Conic8} presents the frequency distribution
of these instances.

{\scriptsize
\begin{table}[htb]
  \caption{Frequency distribution for conics through eight lines}\label{Tab:Conic8}
  \begin{tabular}{|c|c|c|c|c|c|c|c|c|}
  \hline
  \# real & 0 & 2 & 4 & 6 & 8 & 10 & 12 & 14 \\
  \hline
  frequency & 1 & 8 & 26 & 65 & 466 & 1548 & 4765 & 11928 \\
  \hline
  \hline
  \# real & 16 & 18 & 20 & 22 & 24 & 26 & 28 & 30 \\
  \hline
  frequency& 26439 & 52875 & 98129 & 167932 & 270267 & 404918 & 569891 & 756527 \\
  \hline
  \hline
  \# real & 32 & 34 & 36 & 38 & 40 & 42 & 44 & 46 \\
  \hline
  frequency & 942674 & 1114033 & 1246533 & 1332289 & 1355320 & 1319699 & 1226667 & 1091019 \\
  \hline
  \hline
  \# real & 48 & 50 & 52 & 54 & 56 & 58 & 60 & 62 \\
  \hline
  frequency & 932838 & 762463 & 596174 & 449021 & 323927 & 223455 & 149629 & 95740 \\
  \hline
  \hline
  \# real & 64 & 66 & 68 & 70 & 72 & 74 & 76 & 78 \\
  \hline
  frequency & 59141 & 34834 & 19516 & 10672 & 5671 & 2744 & 1290 & 530 \\
  \hline
  \hline
  \# real & 80 & 82 & 84 & 86 & 88 & 90 & 92 & \multicolumn{1}{||c|}{total}\\
  \hline
  frequency & 204 & 90 & 26 & 11 & 3 & 2 & 0 & \multicolumn{1}{||c|}{15662000}\\
  \hline
  \end{tabular}
\end{table}}

%
\subsection{A Schubert problem}\label{Sec:Schubert}

Our last example concerns a problem in the Schubert calculus of enumerative geometry,
which is a rich class of geometric problems involving linear subspaces of a vector space.
Many problems in the Schubert calculus are naturally formulated as overdetermined
polynomial systems.  We investigate one such problem that can also be formulated
as a square polynomial system using the approach of \cite{BHPS}.
In particular, we demonstrate \alphaCert's algorithms for overdetermined systems
as well as investigate a conjecture on the reality of its solutions.

This problem involves four-dimensional linear subspaces (four-planes) $H$ of $\bC^8$ that
have a non-trivial intersection with each of eight general three-planes $K_0,\dotsc,K_7$.
The Schubert calculus predicts $126$ such four-planes.
To formulate this Schubert problem, consider $H$ to be the column space of a $8\times 4$ matrix in
block form
\[
   H\  =\ \left[\begin{array}{c} I_4 \\ X \end{array}\right]\,,
\]
where $I_4$ is the $4\times 4$ identity matrix and $X$ is a $4\times 4$ matrix of
indeterminates.
Represent a three-plane $K$ as the column space of a $8\times 3$ matrix of constants.
Then the condition that $H$ meets $K$ non-trivially is equivalent to the vanishing of the
determinants of the eight $7\times 7$ square submatrices of the $8\times 7$ matrix
 \begin{equation}\label{Eq:A}
   A\  =\ \left[ H~~K \right]\,.
 \end{equation}
In this standard formulation, the Schubert problem is a system of $64$ equations in $16$
indeterminates.
Using a total degree homotopy to solve this would follow $4^{16}$ paths.

There is a second formulation which we used.
Write $K$ in block form,
\[
   K\  =\ \left[\begin{array}{c} \sK_1 \\ \sK_2 \end{array}\right]\,,
\]
where $\sK_1$ and $\sK_2$ are $4 \times 3$ matrices.
A linear dependency among the columns of $A$~\eqref{Eq:A} is given by vectors $v\in\bC^4$ and
$w\in\bC^3$ such that $Hv+Kw=0$.
Applying this to the different blocks of $H$ and $K$ gives
\[
   I_4 v + \sK_1 w\ =\ 0\qquad\mbox{and}\qquad
     X v + \sK_2 w\ =\ 0\,,
\]
which is equivalent to $\widehat{A}w=0$, where $\widehat{A} := \sK_2 - X \sK_1$.
Thus $H$ meets $K$ non-trivially if and only if each $3\times 3$ minor of $\widehat{A}$
vanishes.
This gives a system $F_O(x)$ of $32$ cubic polynomials in $16$ indeterminates, which is more compact
than the original formulation.

We certified solutions to this overdetermined polynomial system $F_O$.
We randomized $F_O$ to maintain the structure of the equations as follows.
For each $i = 0,\dots,7$ and $j = 1,2,3,4$, let $f_{i,j}$ be the determinant
of the submatrix created by removing the $j^{th}$ row of the matrix $\widehat{A}_i$ corresponding to
the $i$th three-plane.
Then, for each $j$, we take four random linear combinations of the
polynomials $f_{0,j},f_{1,j},\dots,f_{7,j}$.
This preserves the multilinear structure of the equations in the
four variable groups corresponding to the columns of $X$.
Solving this system using regeneration \cite{HSW10} finds 22,254 solutions
of which \alphaCertS soft certified using 256-bit precision
that 126 of these are approximate solutions to two different random square subsystems of
$F_O$
with associated solutions within a distance of $\delta = 10^{-10}$ of each other.
The same result was also obtained using $\delta = 10^{-5}$.
Thus, \alphaCertS provided a soft certificate based on the heuristic algorithm
for overdetermined systems that we found all 126 solutions to the Schubert
problem.\smallskip

This Schubert problem has an equivalent formulation as a square system.
The columns of $\widehat{A}$ are linearly dependent if and only if there exists
$0\neq v\in\bC^3$ such that $\widehat{A} v = 0$.
For generic $\alpha_1,\alpha_2\in\bC$, this occurs if and only if there exists
$y_1,y_2\in\bC$ such that
\[
    \widehat{A} \cdot
    \left[\begin{array}{c} y_1 \\ y_2 \\ \alpha_1 y_1 + \alpha_2 y_2 + 1 \end{array}\right]
   \ =\ 0\,.
\]
This yields a system of $32$ polynomials
in $32$ indeterminates, say $F_S(x,y^{(0)},\dots,$ $y^{(7)})$.
This polynomial system consists of four bilinear polynomials
in $x$ and $y^{(i)}$ for each $i = 0,\dots,7$.  Since
$y^{(i)}$ consists of two indeterminates, namely $y_1^{(i)}$ and $y_2^{(i)}$,
a $9$-homogeneous homotopy used to solve $F_S$ would follow $\binom{4}{2}^8 = 6^8$ paths.
As described in \cite{BHPS}, we are interested in the components
of $\Var(F_S)$ having fibers with generic dimension zero.
For generic $K_0,\dots,K_7$, since $\Var(F_S)$ is zero-dimensional,
$\Var(F_O)$ and $\Var(F_S)$ both consist of $126$ isolated points
and $\Var(F_S)$ naturally projects onto $\Var(F_O)$.

We also investigated the number of real solutions when
the three planes $K_i$ are as follows.
For $t\in\bR$, let $\gamma(t)=(1,t,t^2,\dotsc,t^7)\in\bR^8$ be a point on the
moment curve.
Select $24$ rational numbers $t_1,\dotsc,t_{24}$ and for $i=0,\dotsc,7$, let $K_i$ be the
span of the three linearly independent vectors
$\gamma(t_{3i+1})$, $\gamma(t_{3i+2})$, and $\gamma(t_{3i+3})$.
When $t_1<t_2<\dotsb<t_{24}$, the Secant Conjecture~\cite{secant} posits that all 126
solutions will be real, but if the points are not in this or some equivalent order, then
other numbers of real solutions are possible.

Since $K_0,\dots,K_7$ are real, if the constants $\alpha_i$ are real, then the
real points of $\Var(F_O)$ correspond to the real points of
$\Var(F_S)$.
We solved 25000 real instances using random numbers in the interval $[-1,1]$
and softly certified that each had 126 real solutions, using 256-bit precision.

Lastly, we investigated the number of real solutions when
the three planes $K_i$ are as follows.  For $i = 0,\dots,7$, let $t_i\in\bC$ be generic
under the condition that $2k$ are complex conjugate pairs and $8-2k$ are real, where $0 \leq k \leq 4$.
Define $K_i = T(t_i)$ where
\[
  T(t) = \left[\begin{array}{ccc} 1 & 0 & 0 \\ t & 1 & 0 \\ t^2 & 2t & 1 \\
      t^3 & 3t^2 & 3t \\ t^4 & 4t^3 & 6t^2 \\ t^5 & 5t^4 & 10t^3 \\ t^6 & 6t^5 & 15t^4 \end{array}\right].
\]
Then $K_i$ is the three-plane osculating the moment curve at the point $\gamma(t_i)$.
When $k=0$, that is, when each $t_i$ is real, this is the Shapiro Conjecture (MTV
Theorem)~\cite{FRSC,MTV} and all 126 solutions are real.
We tested 1000 such instances and for each, \alphaCertS correctly identified all 126
solutions to be  real.
Our primary interest was when $k>0$, for we wanted to test the hypothesis that there would
be a  lower bound to the number of real solutions if the set of osculating three-planes
were real
(that is, if $\{\overline{K_0},\dotsc,\overline{K_7}\}=\{K_0,\dotsc,K_7\}$).
This is what we found, as can be seen in the partial frequency table we give below.
(To better show the lower bounds, we omit writing $0$ in the cells with no observed
instances.)
This enumeration of real solutions was softly certified using 256-bit precision.

\begin{table}[htb]
  \caption{Frequency distribution for the Schubert problem}\label{Tab:Schubert}
  \begin{tabular}{|c||c|c|c|c|c|c|c|c|c|c|c|c|c|c||r|}
  \hline
    & \multicolumn{12}{|c|}{\# real}& \\
  \hline
   $k$ &0& 2 & 4 & 6 & 8 & 10 & 12 &$\dotsb$& 18 & 20 &  22 &$\dotsb$&124& 126 & total \\
  \hline \hline
    0&&&&&& & &$\dotsb$&& &  &$\dotsb$&& 1000& 1000\\
  \hline
    1  &&&& 6 & 6 & 10 & 88 &$\dotsb$&554&1888&1832&$\dotsb$&69& 2021& 42000\\
  \hline
    2  &&&&&&2614&3771&$\dotsb$&3285&1579& 1378&$\dotsb$&1& 38& 24000\\
  \hline
    3  &&&&&&8896&4479&$\dotsb$&1079& 721& 2586&$\dotsb$&&    & 23500\\
  \hline
    4  &&&&&& &    &$\dotsb$&&&19134&$\dotsb$& &  1& 22500\\
  \hline
  \end{tabular}
\end{table}
This computation was part of a larger test of hypothesized lower bounds~\cite{lower}.

%
\section{Conclusion}\label{Sec:Conclusion}

Smale's $\alpha$-theory provides a way to certify
solutions to polynomial systems, determine
if two points correspond to distinct solutions,
and determine if the corresponding solution is real.
Using either exact rational
or arbitrary precision floating point arithmetic,
\alphaCertS is a program which implements
these $\alpha$-theoretical methods.

We have also produced a Maple interface to \alphaCertS
to facilitate the construction of the input files needed.

\section*{Acknowledgements} The authors would like to thank Mike Shub for his helpful comments
and the first author would like to thank the organizers of
the Foundations of Computational Mathematics thematic program at the Fields Institute.

\providecommand{\bysame}{\leavevmode\hbox to3em{\hrulefill}\thinspace}
\providecommand{\MR}{\relax\ifhmode\unskip\space\fi MR }
\providecommand{\MRhref}[2]{%
  \href{http://www.ams.org/mathscinet-getitem?mr=#1}{#2}
}
\providecommand{\href}[2]{#2}

\end{document}